\documentclass[11pt,a4paper]{article}

\usepackage{amsmath,amsfonts,amssymb,amsthm,latexsym}
\usepackage{graphics,epsfig,color,enumerate}
\usepackage{setspace,mathtools,enumitem,dsfont,verbatim}
\usepackage{caption}
\usepackage{subcaption}
\usepackage{todonotes}
\usepackage{bbm}
\usepackage[english]{babel}
\graphicspath{{./figures/}}
\makeatletter
\def\@captype{figure}
\makeatother

\newtheorem{theorem}{Theorem}[section]

\newtheorem{lemma}[theorem]{Lemma}
\newtheorem{proposition}[theorem]{Proposition}

\newtheorem{definition}[theorem]{Definition}
\newtheorem{cond}[theorem]{Condition}

\numberwithin{equation}{section}

\newcommand{\N}{\mathbb{N}}

\newcommand{\Conf}{\emph{\text{Conf}}}
\newcommand{\sss}{\scriptscriptstyle}
\newcommand{\TV}{\sss \mathrm{TV}}
\newcommand{\vep}{\varepsilon}

\newcommand{\eqn}[1]{\begin{equation} #1 \end{equation}}
\newcommand{\eqan}[1]{\begin{align} #1 \end{align}}
\newcommand{\nn}{\nonumber}
\newcommand{\expec}{\mathbb{E}}
\newcommand{\prob}{\mathbb{P}}

\newcommand{\cstat}{c_{n,\mathrm{stat}}}

\newcommand{\SP}{\mathsf{SP}}
\newcommand{\SA}{\mathsf{SA}}
\newcommand{\pstat}{\mathbb{P}^{\mathrm{stat}}}
\renewcommand{\pmod}{\mathbb{P}^{\mathrm{mod}}}

\newcommand{\dmax}{d_\mathrm{max}}

\newcommand{\dtv}[1]{d_{\sss \rm TV}(#1)}

\newcommand{\probCM}{P_{x,\eta}}

\newcommand{\indic}[1]{\mathbbm{1}_{\{#1\}}}

\newcommand{\CMnd}{{\rm CM}_n(\boldsymbol{d})}

\begin{document}

\title{Random walks on dynamic configuration models:\\ 
a trichotomy}

\author{\renewcommand{\thefootnote}{\arabic{footnote}}
Luca Avena
\footnotemark[1]
\\
\renewcommand{\thefootnote}{\arabic{footnote}}
Hakan G\"{u}lda\c{s}
\footnotemark[1]
\\
\renewcommand{\thefootnote}{\arabic{footnote}}
Remco van der Hofstad
\footnotemark[2]
\\
\renewcommand{\thefootnote}{\arabic{footnote}}
Frank den Hollander
\footnotemark[1]
}

\footnotetext[1]{
Mathematical Institute, Leiden University, P.O.\ Box 9512,
2300 RA Leiden, The Netherlands
}
\footnotetext[2]{
Department of Mathematics and Computer Science, Eindhoven University of Technology, 
P.O.\ Box 513, 5600 MB Eindhoven, The Netherlands
}

\date{\today}


\maketitle

\begin{abstract}
We consider a dynamic random graph on $n$ vertices that is obtained by starting from a random 
graph generated according to the configuration model with a prescribed degree sequence and 
at each unit of time randomly rewiring a fraction $\alpha_n$ of the edges. We are interested in 
the mixing time of a random walk without backtracking on this dynamic random graph in the limit 
as $n\to\infty$, when $\alpha_n$ is chosen such that $\lim_{n\to\infty} \alpha_n (\log n)^2 = \beta 
\in [0,\infty]$. In \cite{AGvdHdH} we found that, under mild regularity conditions on the degree 
sequence, the mixing time is of order $1/\sqrt{\alpha_n}$ when $\beta=\infty$. In the present paper 
we investigate what happens when $\beta \in [0,\infty)$. It turns out that the mixing time is of order 
$\log n$, with the scaled mixing time exhibiting a one-sided cutoff when $\beta \in (0,\infty)$ and a 
two-sided cutoff when $\beta=0$. The occurrence of a one-sided cutoff is a rare phenomenon. In 
our setting it comes from a competition between the time scales of mixing on the static graph, as 
identified by Ben-Hamou and Salez \cite{BHS}, and the regeneration time of first stepping across 
a rewired edge. 
\end{abstract}

{\small \medskip\noindent


\medskip\noindent
{\it Acknowledgment.} 
The work in this paper was supported by the Netherlands Organisation for Scientific Research (NWO) 
through Gravitation-grant NETWORKS-024.002.003. RvdH was also supported by NWO through 
VICI-grant 639.033.806.}



\section{Introduction}
\label{S1}


\subsection{Background}
\label{MotBack}

The goal of the present paper is to study the mixing time of a random walk \emph{without 
backtracking} on a dynamic version of the configuration model. The \emph{static} configuration 
model is a random graph with a prescribed degree sequence. For random walk on the static 
configuration model, with or without backtracking, the asymptotics of the associated mixing 
time, and related properties such as the presence of the so-called cutoff phenomenon, were 
derived recently by Berestycki, Lubetzky, Peres and Sly~\cite{BLPS}, and by Ben-Hamou 
and Salez~\cite{BHS}. In particular, under mild assumptions on the degree sequence, 
guaranteeing that the graph is an \emph{expander} with high probability, the mixing time 
was shown to be of order $\log n$, with $n$ the number of vertices. 

In an earlier paper \cite{AGvdHdH} we consider a \emph{discrete-time dynamic} version of 
the configuration model, where at each unit of time a fraction $\alpha_n$ of the edges is 
sampled and rewired uniformly at random. Our dynamics \emph{preserves the degrees} 
of the vertices. Consequently, when considering a random walk on this dynamic configuration 
model, its \emph{stationary distribution remains constant over time} and the analysis of 
its mixing time is a well-posed question. It is natural to expect that, due to the graph 
dynamics, the random walk \emph{mixes faster} than the $\log n$ order known for 
the static model. Under very mild assumptions on the prescribed degree sequence
(Condition~\ref{cond-degree-reg} below), we have shown that this is indeed the case 
when $(\alpha_n)_{n\in\N}$ satisfies $\lim_{n\to\infty} \alpha_n(\log n)^2 = \infty$, which 
corresponds to a regime of `fast enough' graph dynamics. In particular, we have shown 
that for every $\varepsilon \in(0,1)$ the $\varepsilon$-mixing time grows like $\sqrt{2\log
(1/\varepsilon)/\alpha_n}$ as $n\to\infty$ (when also $\lim_{n\to\infty} \alpha_n = 0$), 
with high probability (in the sense of Definition~\ref{def:whp} below). In the present 
paper we look at a slower dynamics, namely, $(\alpha_n)_{n\in\N}$ satisfying $\lim_{n\to\infty} 
\alpha_n(\log n)^2 = \beta \in [0,\infty)$. Our main result (Theorem~\ref{mainthm} below) 
states that, under somewhat stronger assumptions on the prescribed degree sequence
(Condition~\ref{cond-degree-reg2} below), the mixing time is of order $\log n$, as for the 
static model, but that there is an interesting difference between the cases $\beta \in (0,\infty)$ 
and $\beta=0$.

The rest of the paper is organised as follows. In Section~\ref{Model} we define the model. 
This is a verbatim repetition of what was written in \cite[Section 1.2]{AGvdHdH}, in which 
we introduce notation and set the stage. In Section~\ref{Trichotomy} we state our main 
theorem, which is a \emph{trichotomy} for the cases $\beta=\infty$, $\beta \in (0,\infty)$ 
and $\beta=0$. In Section~\ref{Discussion} we place this theorem in its proper context.

Throughout the sequel we use standard notations for the asymptotic comparison of functions
$f,g\colon\,\N\to[0,\infty)$: $f(n) = O(g(n))$ or $g(n) = \Omega(f(n))$ when $\limsup_{n\to\infty} 
f(n)/g(n) < \infty$; $f(n) = o(g(n))$ or $g(n) = \omega(f(n))$ when $\lim_{n\to\infty}f(n)/g(n) = 0$; 
$f(n) = \Theta(g(n))$ when both $f(n) = O(g(n))$ and $g(n) = O(f(n))$.


\subsection{Model}
\label{Model}

We start by defining the model and setting up the notation. The set of vertices is denoted by 
$V$ and the degree of a vertex $v\in V$ by $d(v)$. Each vertex $v\in V$ is thought of as being 
incident to $d(v)$ \emph{half-edges} (see Fig.~\ref{fig:halfedge}). We write $H$ for the set of 
half-edges, and assume that each half-edge is associated to a vertex via incidence. We denote 
by $v(x)\in V$ the vertex to which $x\in H$ is incident and by $H(v) \coloneqq \{x\in H\colon\,v(x)=v \}
\subset H$ the set of half-edges incident to $v\in V$. If $x,y\in H(v)$ with $x\neq y$, then we 
write $x\sim y$ and say that $x$ and $y$ are siblings of each other. The (forward) degree of a 
half-edge $x\in H$ is defined as
\begin{equation}
\label{degdef}
\deg(x) \coloneqq d(v(x))-1.
\end{equation}
We consider graphs on $n$ vertices, i.e., $|V| = n$, with $m$ edges, so that 
\begin{equation}
|H|=\sum_{v\in V} d(v) = 2m \eqqcolon \ell.
\end{equation}

\begin{figure}[htbp]
\centering
\includegraphics[width=0.25\textwidth]{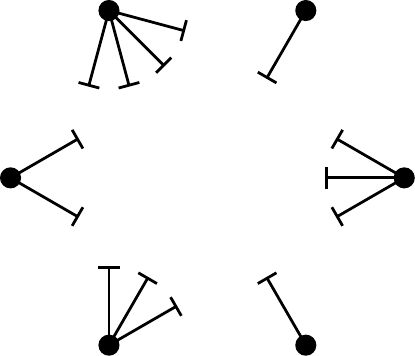}
\vspace{0.2cm}
\caption{\small Vertices with half-edges.}
\label{fig:halfedge}
\end{figure}

The \emph{edges} of the graph will be given by a \emph{configuration} that is a \emph{pairing of 
half-edges}. We denote by $\eta(x)$ the half-edge paired to $x\in H$ in the configuration $\eta$.
A configuration $\eta$ will be viewed as a bijection of $H$ without fixed points and with the property 
that $\eta(\eta(x)) = x$ for all $x\in H$ (also called an involution). With a slight abuse of notation, we 
will use the same symbol $\eta$ to denote the set of pairs of half-edges in $\eta$, so $\{x,y\}\in\eta$ 
means that $\eta(x) = y$ and $\eta(y) = x$. Each pair of half-edges in $\eta$ will also be called an 
edge. The set of all configurations on $H$ will be denoted by $\Conf_H$.

We note that each configuration gives rise to a graph that may contain self-loops (edges having 
the same vertex on both ends) or multiple edges (between the same pair of vertices). On the 
other hand, a graph can be obtained via several distinct configurations.

We will consider asymptotic statements in the sense of $|V| = n \to \infty$. Thus, quantities like
$V, H, d, \deg$ and $\ell$ all depend on $n$. In order to lighten the notation, we often suppress 
$n$ from the notation.


\subsubsection{Dynamic configuration model}
\label{DCM}

We recall the definition of the configuration model, phrased in our notation. The configuration 
model on $V$ with degree sequence $(d(v))_{v\in V}$ is the uniform distribution on $\Conf_H$. 
We sometimes write $d_n=(d(v))_{v\in V}$ when we wish to stress the $n$-dependence of the 
degree sequence. A sample $\eta$ from the configuration model can be generated by taking 
a uniform pairing of the elements of $H$. The resulting configuration $\eta$ gives rise to a 
multi-graph on $V$ with degree sequence $(d(v))_{v\in V}$.
 
We begin by describing the random graph process. It is convenient to take as the state space 
the set of configurations $\Conf_H$. For a fixed initial configuration $\eta$ and fixed $2 \leq k 
\leq m = \ell/2$, the graph evolves as follows (see Fig.~\ref{fig:dcm}):
\begin{enumerate}
\item
At each time $t \in \N$, pick $k$ edges (pairs of half-edges) from $C_{t-1}$ uniformly at random 
without replacement. Cut these edges to get $2k$ half-edges and denote this set of half-edges 
by $R_t$.
\item 
Generate a uniform pairing of these half-edges to obtain $k$ new edges. Replace the $k$ edges 
chosen in step 1 by the $k$ new edges to get the configuration $C_t$ at time $t$.
\end{enumerate}
This process rewires $k$ edges at each step by applying the configuration model sampling 
algorithm restricted to $k$ uniformly chosen edges. Since half-edges are not created or destroyed, 
the degree sequence of the graph given by $C_t$ is the same for all $t\in\N_0$. This gives us a 
Markov chain on the set of configurations $\Conf_H$. For $\eta,\zeta\in\Conf_H$, the \emph{transition 
probabilities} for this Markov chain are given by
\begin{align}
\label{DC}
Q(\eta,\zeta) = Q(\zeta,\eta) \coloneqq
\begin{cases}
\frac{1}{(2k-1)!!}\frac{\binom{m - d_\text{Ham}(\eta,\zeta)}{k - d_\text{Ham}(\eta,\zeta)}}{\binom{m}{k}} 
& \text{if }d_\text{Ham}(\eta,\zeta)\leq k, \\
0 
& \text{otherwise},
\end{cases}
\end{align}
where $d_\text{Ham}(\eta,\zeta) \coloneqq |\eta\setminus\zeta| = |\zeta\setminus\eta|$ is the Hamming 
distance between configurations $\eta$ and $\zeta$, which is the number of edges that appear in $\eta$ 
but not in $\zeta$. The factor $1/(2k-1)!!$ comes from the uniform pairing of the half-edges, while the 
factor $\binom{m - d_\text{Ham}(\eta,\zeta)}{k - d_\text{Ham}(\eta,\zeta)}/\binom{m}{k}$ comes 
from choosing uniformly at random a set of $k$ edges in $\eta$ that contains the edges in 
$\eta\setminus\zeta$. It is easy to see that this Markov chain is irreducible and aperiodic, with
stationary distribution the uniform distribution on $\Conf_H$, denoted by $\text{Conf}_H$, which 
is the distribution of the configuration model.

\begin{figure}[htbp]
\centering
\vspace{0.2cm}
\includegraphics[width=0.5\textwidth]{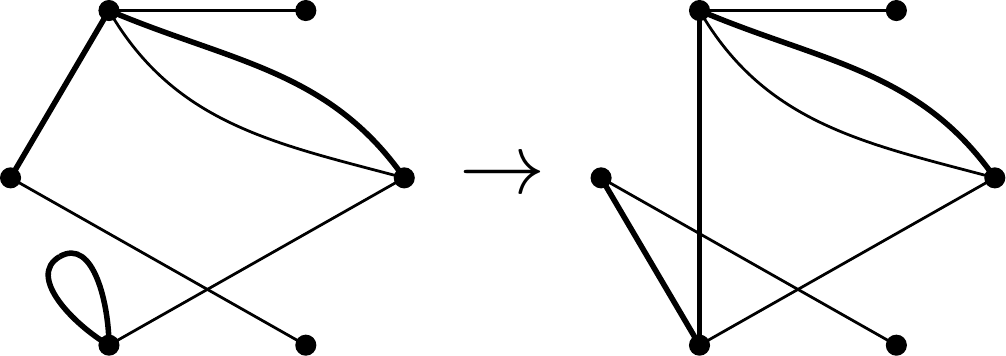}
\vspace{0.2cm}
\caption{\small One move of the dynamic configuration model. Bold edges on the left are the ones 
chosen to be rewired. Bold edges on the right are the newly formed edges.}
\label{fig:dcm}
\end{figure}


\subsubsection{Random walk without backtracking}
\label{NBTRW}

On top of the random graph process we define the random walk without backtracking, i.e., 
the walk cannot traverse the same edge twice in a row. Like Ben-Hamou and Salez~\cite{BHS}, 
we define it as a random walk on the set of half-edges $H$, which is more convenient in the 
dynamic setting because the edges change over time while the half-edges do not.  For a fixed 
configuration $\eta$ and half-edges $x,y\in H$, the transition probabilities of the random walk 
are given by (recall \eqref{degdef})
\begin{align}
\label{NBT}
P_\eta(x,y) \coloneqq
\begin{cases}
\frac{1}{\deg(\eta(x))} & \text{if }\eta(x)\sim y\text{ and }\eta(x)\neq y, \\
0 & \text{otherwise}.
\end{cases}
\end{align}
When the random walk is at half-edge $x$ in configuration $\eta$, it jumps to one of the siblings 
of the half-edge it is paired to uniformly at random (see Fig.~\ref{fig:nbrw_move}). The transition 
probabilities are symmetric with respect to the pairing given by $\eta$, i.e., $P_\eta(x,y) 
= P_\eta(\eta(y),\eta(x))$, in particular, they are doubly stochastic, and so the uniform distribution 
on $H$, denoted by $U_H$, is stationary for $P_\eta$ for any $\eta\in\Conf_H$.

\begin{figure}[htbp]
\centering
\includegraphics[width=0.25\textwidth]{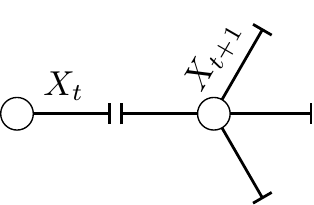}
\caption{\small The random walk moves from half-edge $X_t$ to half-edge $X_{t+1}$, one of 
the siblings of the half-edge $\eta(X_t)$ that $X_t$ is paired to.}
\label{fig:nbrw_move}
\end{figure}

\subsubsection{Random walk on dynamic configuration model}

The random walk without backtracking on the dynamic configuration model is the joint Markov 
chain $(M_t)_{t\in\N_0} = (C_t,X_t)_{t\in\N_0}$ in which $(C_t)_{t\in\N_0}$ is the Markov chain 
on the set of configurations $\Conf_H$ as described in~\eqref{DC}, and $(X_t)_{t\in\N_0}$ 
is the random walk that at each time step $t$ jumps according to the transition probabilities 
$P_{C_t}(\cdot,\cdot)$ as in~\eqref{NBT}. 

Formally, for initial configuration $\eta$ and half-edge $x$, the one-step evolution of the joint 
Markov chain is given by the conditional probabilities
\begin{align}
\prob_{\eta,x}\big(C_t = \zeta, X_t = z \mid C_{t-1} = \xi, X_{t-1} = y\big) 
= Q(\xi,\zeta)\,P_\zeta(y,z),
\qquad t \in \N,
\end{align}
with
\begin{align}
\prob_{\eta,x}(C_0 = \eta,X_0=x) = 1.
\end{align}
Thus, we first rewire the graph, and afterwards let the random walk make a non-backtracking step 
in the updated configuration. It is easy to see that if $d(v)>1$ for all $v\in V$, then this Markov chain 
is irreducible and aperiodic, and has the unique stationary distribution $\text{Conf}_H \times U_H$. 

While the graph process $(C_t)_{t\in\N_0}$ and the joint process $(M_t)_{t\in\N_0}$ are Markovian, 
the random walk $(X_t)_{t\in\N_0}$ is \emph{not}. However, $U_H$ is still the stationary distribution 
of $(X_t)_{t\in\N_0}$. Indeed, for any $\eta\in\Conf_H$ and $y \in H$, we have
\begin{equation}
\sum_{x\in H}U_H(x)\,\prob_{\eta,x}(X_t = y) 
= \sum_{x\in H} \frac{1}{\ell}\,\prob_{\eta,x}(X_t = y) = \frac{1}{\ell} = U_H(y).
\end{equation}
The next to last equality uses that $\sum_{x\in H}\prob_{\eta,x}(X_t = y)=1$ for every $y\in H$,
which can be seen by conditioning on the graph process and using that the space-time 
inhomogeneous random walk has a doubly stochastic transition matrix (recall the remarks 
made below \eqref{NBT}).


\subsection{A trichotomy}
\label{Trichotomy}

We are interested in the behaviour of the total variation distance between the distribution of 
$X_t$ and the uniform distribution
\begin{equation}
\mathcal{D}_{\eta,x}(t) \coloneqq \|\prob_{\eta,x}(X_t\in\cdot\,)-U_H(\cdot)\|_{\TV}.
\end{equation}
Note that $\mathcal{D}_{\eta,x}(t)$ depends on the initial configuration $\eta$ and half-edge $x$. 
We will prove statements that hold for \emph{typical} choices of $(\eta,x)$ under the uniform 
distribution $\mu_n$ (recall that $H$ depends on the number of vertices $n$) given by
\begin{equation}
\label{mu}
\mu_n := \text{Conf}_{H} \times U_H \quad \text{ on  }  \Conf_H \times H,
\end{equation} 
where {\em typical} is made precise through the following definition:

\begin{definition}[{\bf With high probability}]
\label{def:whp}
A statement that depends on the initial configuration $\eta$ and initial half-edge $x$ is said to hold 
{\em with high probability (whp)} in $\eta$ and $x$ if the $\mu_n$-measure of the set of pairs $(\eta,x)$ 
for which the statement holds tends to $1$ as $n\to\infty$. 
\end{definition}


\subsubsection{Regularity conditions}

In Theorem~\ref{mainthm} below we use two sets of regularity conditions on the degree sequence:

\begin{cond}
\label{cond-degree-reg} {\bf (Regularity of degrees)}
\begin{itemize}
\item[{\rm (R1)}]
$\ell$ is even and $\ell=\Theta(n)$ as $n\to\infty$.
\item[{\rm (R2)}]
$\limsup_{n\to\infty} \nu_n < \infty$, where 
	\begin{equation}
	\nu_n \coloneqq \frac{\sum_{z\in H}\deg(z)}{\ell}=\frac{\sum_{v\in V} d(v)[d(v)-1]}{\sum_{v\in V} d(v)}
	\end{equation} 
denotes the expected degree of a uniformly chosen half-edge. 
\item[{\rm (R3)}] 
$d(v)\geq 2$ for all $v\in V$.
\end{itemize}
\end{cond}

\begin{cond}
\label{cond-degree-reg2} {\bf (Regularity of degrees)}
\begin{itemize}
\item[{\rm (R1*)}]
$\dmax = \ell^{o(1)}$ as $n\to\infty$, where
	\begin{equation}
	\dmax \coloneqq \max_{v\in V}d(v).
	\end{equation} 
\item[{\rm (R2*)}]
As $n\to\infty$,
	\begin{equation}
	\frac{\lambda_2}{\lambda_1^3} = \omega\left(\frac{(\log\log\ell)^2}{\log\ell}\right),
	\qquad \frac{\lambda_2^{3/2}}{\lambda_3\sqrt{\lambda_1}} 
	= \omega\left(\frac{1}{\sqrt{\log\ell}}\right),
	\end{equation}
where
	\begin{equation}
	\lambda_1 \coloneqq \frac{1}{\ell} \sum_{z\in H} \log(\deg(z)), \qquad 
	\lambda_m \coloneqq \frac{1}{\ell} \sum_{z\in H} |\log(\deg(z))-\lambda_1|^m, \quad m = 2,3.
	\end{equation}
\item [{\rm (R3*)}] 
$d(v)\geq 3$ for all $v\in V$.
\end{itemize}
\end{cond}

\noindent
Condition~{\rm \ref{cond-degree-reg}} was used in \cite{AGvdHdH} to deal with the regime
of `fast graph dynamics'. Conditions (R1) and (R2) are minimal requirements to guarantee 
that the graph is locally tree-like. Condition (R3) ensures that the random walk without 
backtracking is well-defined. Condition~{\rm \ref{cond-degree-reg2}} was used in \cite{BHS} 
to deal with the regime of no graph dynamics, i.e., the static graph. Condition (R1*) provides 
control on the large degrees. Condition (R2*) is technical and states that the degrees vary neither 
too little nor too much. Condition (R3*) ensures that the graph is connected with high probability
and that there are no nodes where the random walk without backtracking moves deterministically. 

Below, we will work under the Conditions (R1)--(R3) as well as (R1*)--(R3*). If $D_n = d(V_n)$ 
denotes the degree of a random vertex, then Condition (R2*) is implied by the often used 
condition that $D_n \to D$ in distribution (when $\prob(D\geq 3)>0$), together with $\expec[D_n]
\to \expec[D]$ (see e.g.\ van der Hofstad~\cite[Chapter 7]{RvdH1}). Thus, Condition (R2*) is 
rather mild. Condition (R1*) excludes vertices with a degree that is a positive power of $n$,
which is claimed to be realistic for real-world networks (see e.g.\ \cite[Chapter 1]{RvdH1} for 
an extensive introduction). We have a truncation argument, along the lines of the one in 
Berestycki, van der Hofstad and Salez~\cite{BvdHS}, showing that the degrees can be 
truncated and the random walk is unlikely to notice this truncation. However, the truncated 
graph may have vertices of degree 2, so that it is not clear how to apply the results in Ben-Hamou
and Salez~\cite{BHS}. Furthermore, we believe that Condition (R3*) is unnecessary for our 
results. We state it here because we rely on the work of \cite{BHS}, who consider random walk 
without backtracking started from the worst-possible starting point. When there is a positive proportion 
of vertices of degree $2$, the configuration model is bound to contain a long path of such 
vertices. On such a stretch, the walk moves deterministically, but it slows down the mixing
because it takes time $\omega(\log{n})$ to leave the stretch. Thus, mixing would occur at 
a time that is $\omega(\log{n})$ larger than that when the walk starts from a uniform vertex, 
which makes worst-case and average-case mixing different. Still, since our walk starts from 
the uniform measure on half-edges, it is unlikely to encounter such a stretch. We refrain 
from investigating this issue further.


\subsubsection{Main theorem}

Define the proportion of rewired edges per unit of time as
\begin{equation}
\label{ratio}
\alpha_n:=k/m, \qquad n \in \N,
\end{equation}
where $m = \ell/2$ is the total number of edges and $k$ is the number of edges that get rewired 
per unit of time. For the static model ($\alpha_n \equiv 0$), under Condition~\ref{cond-degree-reg2}, 
the $\varepsilon$-mixing time $\inf\{t\in\N_0\colon\,\mathcal{D}_{\eta,x}(t) \leq \varepsilon\}$ is known 
to scale like $[1+o(1)]\,\cstat\log n$ for all $\varepsilon \in (0,1)$, with $\cstat = 1/\lambda_1 \in 
(0,\infty)$ (Ben-Hamou and Salez~\cite{BHS}). If Condition~\ref{cond-degree-reg} holds too, then 
$n \mapsto \cstat$ is bounded away from $0$ and $\infty$. If also the degree distribution tends to 
a limit, then $\lim_{n\to\infty} \cstat = c_{\mathrm{stat}} \in (0,\infty)$.

Our main theorem shows that the above behaviour turns into a \emph{trichotomy} for the dynamic 
model:  

\begin{theorem}[{\bf Scaled mixing profiles}]
\label{mainthm}
Suppose that $\lim_{n\to\infty}\alpha_n(\log n)^2 = \beta\in[0,\infty]$. The following hold whp in 
$\eta$ and $x$:
\begin{itemize}
\item[{\rm (1)}] 
Subject to Condition~{\rm \ref{cond-degree-reg}}, if $\beta=\infty$, then
\begin{equation}
\mathcal{D}_{\eta,x}\big(c\alpha_n^{-1/2}\,\big) 
= \mathrm{e}^{-c^2/2} + o(1), \quad c \in [0,\infty).
\end{equation}
\item[{\rm (2)}] 
Subject to Condition~{\rm \ref{cond-degree-reg}(R1)} and Condition~{\rm \ref{cond-degree-reg2}}, 
if $\beta \in (0,\infty)$, 
then
\begin{equation}
\label{equ:critical}
\mathcal{D}_{\eta,x}\big(c \log n\,\big) 
= \left\{\begin{array}{ll}
\mathrm{e}^{-\beta c^2/2} + o(1), &c \in [0,\cstat),\\
o(1), &c \in (\cstat,\infty).
\end{array}
\right.
\end{equation}
\item[{\rm (3)}] 
Subject to Condition~{\rm \ref{cond-degree-reg}(R1)} and Condition~{\rm \ref{cond-degree-reg2}}, 
if $\beta = 0$, then
\begin{equation}
\label{equ:subcritical}
\mathcal{D}_{\eta,x}\big(c \log n\,\big) 
= \left\{\begin{array}{ll}
1-o(1), &c \in [0,\cstat),\\
o(1), &c \in (\cstat,\infty).
\end{array}
\right.
\end{equation}
\end{itemize}
\end{theorem}

\begin{figure}[htbp]
\begin{center}
\hspace{-.85cm}
\includegraphics[width=0.55\textwidth]{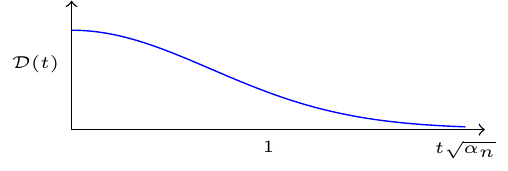}
\includegraphics[width=0.6\textwidth]{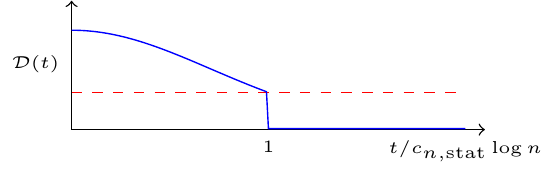}
\includegraphics[width=0.6\textwidth]{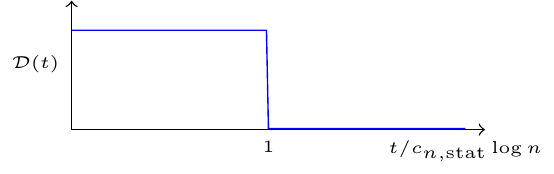}
\caption{\small Plot of $\mathcal{D}(t)$ on time scale $1/\sqrt{\alpha_n}$ for $\beta=\infty$,
respectively, on time scale $\cstat\log n$ for $\beta \in (0,\infty)$ and $\beta = 0$. Because
the scaling holds whp in $\eta$ and $x$, we have suppressed these indices.}
\label{fig:tvdisttrichotomy}
\end{center}
\end{figure}

\noindent
The proof of Theorem~\ref{mainthm} is organised as follows. Theorem~\ref{mainthm}(1) was already
proved in \cite{AGvdHdH}. In Section~\ref{Stopping} we show that Theorems~\ref{mainthm}(2)--(3) 
follow from a key proposition (Proposition~\ref{prop:mainprop} below), which will be proved in 
Sections~\ref{SAtra}--\ref{mainprop}. In Section~\ref{SAtra} we show that on scale $\log n$ 
with high probability the random walk is self-avoiding, i.e., does not visit the same vertex twice, 
and that the same holds for a version of the random walk with random resets. In Section~\ref{mainprop} 
we compute probabilities of rewiring histories and of self-avoiding paths conditional on rewiring 
histories.


\subsection{Discussion}
\label{Discussion}

{\bf 1.}
Theorem~\ref{mainthm} gives the sharp asymptotics of the mixing profiles in three 
regimes, which we refer to as \emph{supercritical} ($\beta=\infty$), \emph{critical} ($\beta 
\in (0,\infty)$) and \emph{subcritical} ($\beta=0$). The latter includes the case of the static
configuration model. While in the supercritical regime the mixing time is of order  
$1/\sqrt{\alpha_n} = o(\log n)$, in the critical and the subcritical regime it is of order
$\log n$ (see Fig.~\ref{fig:tvdisttrichotomy}). Note that for $\beta=\infty$ the scaling
does not depend on the degrees, while for $\beta \in [0,\infty)$ it does via the constant
$\cstat$.  
 
\medskip\noindent 
{\bf 2.}
For the static model, because the scaling of the $\varepsilon$-mixing time does not depend 
on $\varepsilon \in (0,1)$ (Ben-Hamou and Salez~\cite{BHS}) there is \emph{two-sided cutoff}, i.e., 
the total variation distance drops from $1$ to $0$ in a time window of width $o(\log n)$. 
Theorem~\ref{mainthm} shows that this behaviour persists throughout the subcritical regime, 
but that in the critical regime the drop is not from height $1$ but from height $<1$, i.e., 
there is \emph{one-sided cut-off}. In contrast, in the supercritical regime there is \emph{no 
cutoff}, i.e., the total variation distance drops from $1$ to $0$ gradually on scale $1/\sqrt{\alpha_n}$. 
 
\medskip\noindent
{\bf 3.}
We emphasize that we look at the mixing times for `typical' initial conditions and we look 
at the distribution of the random walk averaged over the trajectories of the graph process:
the `annealed' model. It would be interesting to look at different setups, such as `worst-case' 
mixing, in which the maximum of the mixing times over all initial conditions is considered, 
or the `quenched' model, in which the entire trajectory of the graph process is fixed instead 
of just the initial configuration. In such setups the results can be drastically different. 

\medskip\noindent
{\bf 4.} 
It would be of interest to extend our results to random walk \emph{with backtracking}. This 
is much harder. Indeed, because the configuration model is locally tree-like and random 
walk without backtracking on a tree is the same as self-avoiding walk, in our proof we can 
exploit the fact that typical walk trajectories are self-avoiding. In contrast, for the random 
walk with backtracking, after it jumps over a rewired edge, which in our model serves as a 
randomized stopping time, it may jump back over the same edge, in which case it has not 
mixed. This problem remains to be resolved.


\section{Stopping time decomposition}
\label{Stopping}

As in \cite{AGvdHdH}, the proof is based on a randomized stopping time argument.
Let
\begin{equation}
\tau \coloneqq \min\{t\in\N\colon\, X_{t-1}\in R_{\leq t}\}.
\end{equation}
where $R_{\leq t} \coloneqq \cup_{s=1}^t R_s$. By the triangle inequality, 
we have
\begin{align}
\label{equ:tvuppertriangle}
\mathcal{D}_{\eta,x}(t) 
&\leq \prob_{\eta,x}(\tau > t)\|\prob_{\eta,x}(X_t\in\cdot\mid \tau >t) - U_H\|_{\TV} \nn \\
&\qquad+ \prob_{\eta,x}(\tau \leq t)\|\prob_{\eta,x}(X_t\in\cdot\mid \tau \leq t) - U_H\|_{\TV}
\end{align}
and
\begin{align}
\label{equ:tvlowertriangle}
\mathcal{D}_{\eta,x}(t) 
&\geq \prob_{\eta,x}(\tau > t)\|\prob_{\eta,x}(X_t\in\cdot\mid \tau >t) - U_H\|_{\TV}\nn \\
&\qquad- \prob_{\eta,x}(\tau \leq t)\|\prob_{\eta,x}(X_t\in\cdot\mid \tau \leq t) - U_H\|_{\TV}.
\end{align}

\begin{proposition}[\bf{Closeness to stationarity and tail behavior of stopping time}]
\label{prop:mainprop}
$\mbox{}$\\
Suppose that Condition~{\rm \ref{cond-degree-reg}(R1)} and Condition~{\rm \ref{cond-degree-reg2}} 
hold and that $\beta \in [0,\infty)$. If $t = t(n) = [1+o(1)]\,c\log n$ for some $c\in(0,\infty)$, then whp 
in $\eta$ and $x$,
\begin{align}
&\|\prob_{\eta,x}(X_t\in\cdot\mid \tau > t) - U_H(\cdot)\|_{\TV} 
= \begin{cases}
1-o(1),  &c \in [0,\cstat), \\
o(1),     &c \in (\cstat,\infty), 
\label{equ:unstoppedtv}
\end{cases} \\
&\prob_{\eta,x}(\tau > t) = (1-\alpha_n)^{t(t+1)/2}+o(1). 
\label{equ:tautailprob}
\end{align}
If, in addition, $k = k(n) = \omega((\log n)^2)$, then
\begin{equation}
\|\prob_{\eta,x}(X_t\in\cdot\mid \tau\leq t) - U_H(\cdot)\|_{\TV} = o(1). 
\label{equ:stoppedtv}
\end{equation}
\end{proposition}

We show how Theorems~\ref{mainthm}(2)--(3) follow from Proposition~\ref{prop:mainprop}:

\begin{proof}[Proof of Theorem~{\rm \ref{mainthm}(2)--(3)}]
First we prove \eqref{equ:critical}. Under the condition $\lim_{n\to\infty}\alpha_n(\log n)^2 
= \beta \in (0,\infty)$, since $m = \Theta(n)$ we have $k = \omega((\log n)^2)$, and so 
we can use all three items of Proposition~\ref{prop:mainprop}. From \eqref{equ:tvuppertriangle},
\eqref{equ:tvlowertriangle} and \eqref{equ:stoppedtv} it follows that, for any $t = [1+o(1)]\,c\log n$,
\begin{equation}
\mathcal{D}_{\eta,x}(t) = \prob_{\eta,x}(\tau > t)\|\,
\prob_{\eta,x}(X_t\in\cdot\mid \tau >t) - U_H\|_{\TV} + o(1).
\end{equation}
Since $\lim_{n\to\infty} \alpha_n = 0$ and $t\alpha_n = o(1)$, by \eqref{equ:tautailprob} also
\begin{equation}\label{tailasymp}
\prob_{\eta,x}(\tau > t) = (1-\alpha_n)^{t(t+1)/2} + o(1) = \exp(-\alpha_n t^2/2) + o(1).
\end{equation}
Since $\alpha_n = [1+o(1)]\,\beta/(\log n)^2$, \eqref{tailasymp} together with \eqref{equ:unstoppedtv} 
gives us
\begin{equation}
\mathcal{D}_{\eta,x}(t) =
\begin{cases}
\exp(-\beta c^2/2) + o(1), &c\in[0,\cstat), \\
o(1), &c\in(\cstat,\infty).
\end{cases}
\end{equation}
Next, we prove \eqref{equ:subcritical}. If $\lim_{n\to\infty} \alpha_n(\log n)^2 = \beta = 0$, then by 
\eqref{equ:tautailprob}, for any $t = [1+o(1)]\,c\log n$,
\begin{equation}\label{tailsubcrit}
\prob_{\eta,x}(\tau > t) = \exp(-\alpha_nt^2/2) + o(1) = 1 - o(1),
\qquad \prob_{\eta,x}(\tau \leq t) = o(1).
\end{equation}
Inserting \eqref{tailsubcrit} into \eqref{equ:tvuppertriangle} and \eqref{equ:tvlowertriangle}, we get
\begin{equation}
\mathcal{D}_{\eta,x}(t) = [1-o(1)]\,\|\prob_{\eta,x}(X_t\in\cdot\mid \tau >t) - U_H\|_{\TV} + o(1).
\end{equation}
Using \eqref{equ:unstoppedtv}, we obtain
\begin{equation}
\mathcal{D}_{\eta,x}(t) =
\begin{cases}
1-o(1), &c\in[0,\cstat), \\
o(1), &c\in(\cstat,\infty).
\end{cases}
\end{equation}
\end{proof}


\section{Self-avoiding trajectories}
\label{SAtra}

In this section, we show that the random walk trajectories are self-avoiding on the relevant 
time scales with high probability. We let $\SA_t$ denote the event $\{v(X_s)\neq v(X_{s'})
\text{ for any }0\leq s < s' \leq t\}$, i.e., no two half-edges are incident to the same vertex 
along the trajectory up to time $t$.

Along the way we need a random walk on the static model that is a slightly modified 
version of the random walk without backtracking. This version will be instrumental in 
the proof of our main theorem. For fixed $t\in\N$, define the $t$-step \emph{modified 
random walk} starting from configuration $\eta$ and half-edge $x$ as follows:
\begin{enumerate}
\item 
Let $\mathcal{T}$ be a random subset of $[t]$ drawn according to a probability mass 
function $(p_t(T))_{T\subset[t]}$ with $p_t(\varnothing) \in (0,1)$ for all $t$ (to be defined 
later on).
\item 
At each time $s \in [t]$, if $s\not\in \mathcal{T}$, then the random walk makes a non-backtracking 
move in configuration $\eta$, while if $s\in\mathcal{T}$, then it jumps to a uniformly chosen 
half-edge (possibly the half-edge it is on).
\end{enumerate}
This is a random walk without backtracking that resets its position to a uniformly chosen 
half-edge at certain random times. We denote its law by $\pmod_{\eta,x}$, and put 
$\pmod_{\eta,x}(X_0 = x) = 1$. Note that, although the distribution of this random walk 
depends on $t$ and on the distribution of $\mathcal{T}$, we suppress these from the 
notation.

If we condition on the event that $\mathcal{T}\neq\varnothing$, then the modified random 
walk makes a uniform jump at some time in $[t]$ after which it becomes stationary, and so
\begin{equation}
\label{equ:mod_stat}
\pmod_{\eta,x}(X_t\in\cdot\mid\mathcal{T} \neq \varnothing) = U_H(\cdot).
\end{equation}
On the other hand, if we condition on the event that $\mathcal{T}=\varnothing$, then 
the modified random walk is the same as the random walk without backtracking on 
the static graph given by configuration $\eta$ starting from $x$. Denoting the 
law of the latter by $\pstat_{\eta,x}$, we have
\begin{equation}
\label{equ:modeqstat}
\pmod_{\eta,x}(\cdot\mid\mathcal{T}=\varnothing) = \pstat_{\eta,x}(\cdot).
\end{equation}

The main result of this section is the following lemma: 

\begin{lemma}
Suppose that Condition~{\rm \ref{cond-degree-reg}(R1)} and Condition~{\rm \ref{cond-degree-reg2}(R1*)} 
hold and that $t=[1+o(1)]\,c\log n$ for some $c\in(0,\infty)$. Then whp in $\eta$ and $x$,
\begin{equation}
\label{equ:sa_prob}
\prob_{\eta,x}(\SA_t) = 1-o(1), \qquad \pmod_{\eta,x}(\SA_t) = 1-o(1).
\end{equation}
\end{lemma}

\begin{proof}
The proof uses two exploration processes on the graph with the help of the two random walks 
in the annealed setting. Recall that $\mu_n = \Conf_H\times U_H$. The annealed measures 
for the two random walks are defined as
\begin{equation}
\prob(\cdot) := \sum_{\eta,x} \mu_n(\eta,x)\,\prob_{\eta,x}(\cdot),
\qquad \pmod(\cdot) := \sum_{\eta,x} \mu_n(\eta,x)\,\pmod_{\eta,x}(\cdot).
\end{equation}

First, we describe the exploration process for the random walk on the dynamic configuration model. 
To compute the probability of a self-avoiding path, we keep track of already explored half-edges. 
The exploration process proceeds as follows:
\begin{enumerate}
\item 
At time $s=0$, choose $x$ uniformly at random from $H$, set $X_0 = x$ and set $A_0$ to be the 
set containing $x$ and all its siblings (the set of `active' half-edges at time $0$).
\item 
At each time $s\in[t]$, reveal the pair of $X_{s-1} = x_{s-1}$ in $C_s$, say $y_{s-1}$. Denote the 
edge $\{x_{s-1},y_{s-1}\}$ by $e_s$. Add $y_{s-1}$ and all its siblings to $A_{s-1}$ to obtain 
$A_s$ (the set of `active' half-edges at time $s$); some siblings may already have been added 
in a previous step.
\item 
Choose one of the siblings of $y_{s-1}$ uniformly at random, say $x_s$, and set $X_s=x_s$.
\end{enumerate}
This procedure builds up the trajectory of the random walk while ignoring what happens in the 
rest of the graph. Note that we only pair the half-edges along the trajectory, while the siblings 
of the half-edges along the trajectory are not paired until they are visited by the random walk.

Under this construction, the first time the random walk is not self-avoiding is the first time the 
revealed pair at step 2 is in the set of active half-edges. Hence we want to bound the probability
\begin{equation}
\prob\big(C_s(x_{s-1}) \in A_{s-1}\mid e_i\in C_i, i \in [t-1]\big),
\end{equation}
where $e_1,\dots,e_{s-1}$ form a self-avoiding path. For any $y\in H\setminus\{x_{s-1}\}$, if $y$
is not paired up to time $s$, then it can be paired to $x_{s-1}$ through the initial pairing at time 0 
or through rewiring at later times. Since the initial pairing is uniform and this distribution is stationary 
under the graph dynamics, for all such $y$ the above conditional probability is the same, and 
so we have
\begin{equation}
\prob\big(C_s(x_{s-1})=y \mid e_i\in C_i, i \in [s-1]\big) \leq \frac{1}{\ell-2s+1},
\end{equation}
where we note that $2(s-1)$ half-edges are paired before time $s$. On the other hand, if $y \in 
H\setminus\{x_{s-1}\}$ is already paired before time $s$, then it can be paired to $x_{s-1}$ only
through rewiring. Hence the same probability is less than it is for an unpaired half-edge, and so 
we have the same upper bound. Summing over $A_{s-1}$, we get
\begin{equation}
\prob\big(C_s(x_{s-1})\in A_{s-1}\mid e_i\in C_i, i \in [s-1]\big) 
\leq \frac{|A_{s-1}|-1}{\ell-2s+1} \leq \frac{s\dmax}{\ell-2s+1},
\end{equation}
where we use that at each time we activate at most $\dmax = \max_{v \in V} d(v)$ half-edges.
Finally, since $\dmax = n^{o(1)}$ by Condition~{\rm \ref{cond-degree-reg2}}(R1*), $t = [1+o(1)]\,
c\log n$ and $\ell = \Theta(n)$ by Condition~{\rm \ref{cond-degree-reg}}(R1), via a union bound 
and summing over $s\in[t]$, we get
\begin{equation}\label{saupperbd}
\prob(\SA_t^c) \leq \frac{\dmax t(t+1)/2}{\ell-2t+1} = o(1).
\end{equation}
Indeed, by the Markov inequality, for any $(w_n)_{n\in\N}$ that tends to zero arbitrarily slow we have
\begin{equation}
\mu_n(\prob_{\eta,x}(\SA_t^c)>w_n) \leq \frac{\prob(\SA_t^c)}{w_n},
\end{equation}
which implies that, with $\mu_n$-probability at least $1-o(1)$,
\begin{equation}
\prob_{\eta,x}(\SA_t) = 1-o(1).
\end{equation}

Next, we describe the exploration process for the modified random walk. Again, we let $A_t$ 
denote the set of active half-edges. Now, instead of random rewirings, we have a static 
configuration chosen randomly according to the configuration model, and we have a set 
of random times $\mathcal{T}\subset[t]$ at which the random walk makes uniform jumps. 
The exploration process proceeds as follows:
\begin{enumerate}
\item 
At time $s=0$, choose $x$ uniformly at random from $H$, set $X_0 = x$ and set $A_0$ to 
be the set containing $x$ and all its siblings. Choose also $\mathcal{T}\subset [t]$ randomly 
with probabilities $(p_t(T))_{T\in[t]}$.
\item 
At each time $t\in\N$:
\begin{enumerate}
\item 
If $s\in[t]\setminus \mathcal{T}$, then reveal the pair of $X_{s-1} = x_{s-1}$ in $\eta$, say 
$y_{s-1}$. Add $y_{s-1}$ and all its siblings to $A_{s-1}$ to obtain $A_s$. Choose one of 
the siblings of $y_{s-1}$ uniformly at random, say $x_t$, and set $X_s=x_s$.
\item 
If $s\in \mathcal{T}$, then choose $x_s$ uniformly at random from $H$, set $X_s = x_s$, 
add $x_s$ and all its siblings to $A_{s-1}$ to obtain $A_s$. 
\end{enumerate} 
\end{enumerate}
Under this construction, the first time the random walk is not self-avoiding is the first time 
we either have that the revealed pair at step 2(a) is in the set of active half-edges or the 
random walk jumps to an active half-edge at step 2(b). We look at the probability 
\begin{equation}
\pmod(X_s\in A_{s-1}\mid X_{[0,s-1]}=x_{[0,s-1]}),
\end{equation}
where $x_{[0,s-1]}$ is a self-avoiding segmented path. We see that if $s\in\mathcal{T}$,
then this probability is $|A_{s-1}|/\ell$. Otherwise it is at most $(|A_{s-1}|-1)/(\ell-2s+1)$, 
and so we get
\begin{equation}
\pmod\big(X_s\in A_{s-1}\mid X_{[0,s-1]}=x_{[0,s-1]}\big) \leq \frac{|A_{s-1}|}{\ell-2s+1}
\leq \frac{s\dmax}{\ell-2s+1}.
\end{equation} 
This bounds agrees with \eqref{saupperbd}, so we get the same conclusion for $\pmod$. 
Hence, with $\mu_n$-probability at least $1-o(1)$,
\begin{equation}
\pmod_{\eta,x}(\SA_t) = 1-o(1).
\end{equation}
\end{proof}

The proof for the modified random walk can be easily adapted to the random walk without backtracking
on the static graph, simply by removing step 2(b) in the exploration process for the modified random 
walk. Hence we also have, whp in $\eta$ and $x$,
\begin{equation}
\label{equ:sa_prob_stat}
\pstat_{\eta,x}(\SA_t) = 1-o(1).
\end{equation}


\section{Proof of the main proposition}
\label{mainprop}

In this section, we prove Proposition~\ref{prop:mainprop}. We use the notation introduced in 
\cite{AGvdHdH} and recall some of the definitions that are needed along the way. 

For a fixed sequence of half-edges $x_{[0,t]}$ with $x_0 = x$ and a fixed set of times $T 
\subseteq [t]$, we use the short-hand notation 
\begin{equation}
A(x_{[0,t]}; T) \coloneqq \big\{x_{i-1} \in R_{\leq i}\,\,\forall\, i\in T, \,x_{j-1}
\not\in R_{\leq j}\,\,\forall\, j\in [1,t] \setminus T\big\},
\end{equation}
where $R_{\leq i}$ denotes the set of half-edges that are rewired up to time $i$. This event 
gives us the rewiring history for the sequence of half-edges $x_{[0,t]}$. More precisely, 
it is the event that for $i\in[t]\setminus T$ the half-edge $x_{i-1}$ in not rewired until time $i$, and 
for $i\in T$ the half-edge $x_{i-1}$ is rewired at some time before or at time $i$.

We say that a sequence $x_{[0,t]}$ of half-edges of length $t$ is a \emph{self-avoiding segmented 
path} in the configuration $\eta$ with respect to $T =\{t_1,\dots,t_r\}\subset[t]$ if $x_{[0,t]}$ is 
self-avoiding, meaning that no two half-edges in $x_{[0,t]}$ are siblings, and each subsequence 
$x_{[t_{i-1},t_i-1]}$ induces a path in $\eta$ for $i\in[r+1]$ with $t_0 = 0$ and $t_{r+1} = t+1$. We 
denote by $\SP_t^\eta(x,y;T)$ the set of all self-avoiding segmented paths in $\eta$ with respect 
to $T$ with $x_0 = x$ and $x_t = y$ (see Fig.~\ref{sasegp}) and by $\SP_t^\eta(x;T)$ the set of 
all self-avoiding segmented paths in $\eta$ with respect to $T$ with $x_0 =x$. Note that for 
$T = \varnothing$ these are simply the sets of self-avoiding paths.

\begin{figure}[htbp]
\centering
\includegraphics[width=0.4\textwidth]{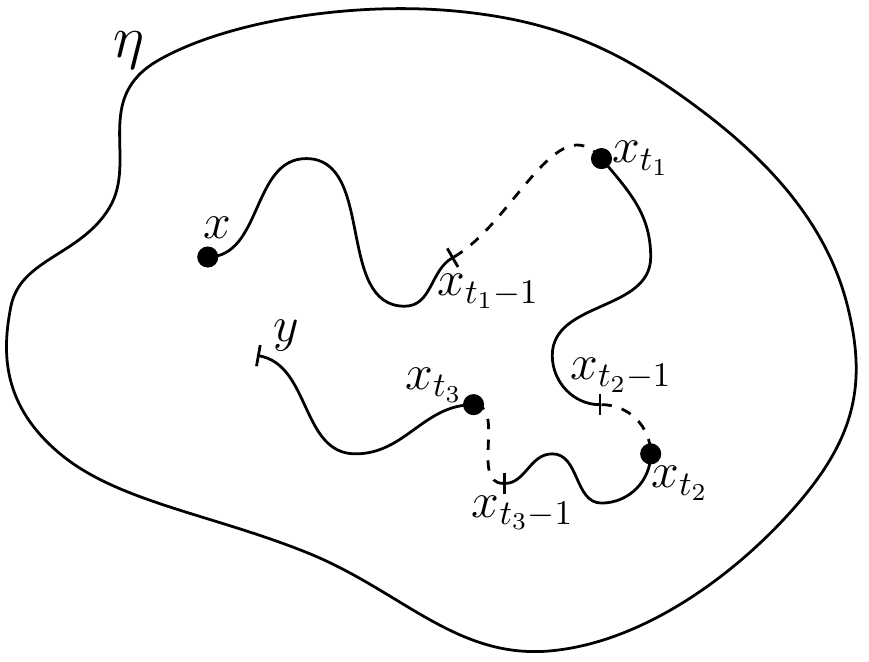}
\caption{\small An element of $\SP_t^\eta(x,y;T)$ with $T=\{t_1,t_2,t_3\}$.}
\label{sasegp}
\end{figure}

Lemmas~\ref{lem:equalpathprobs} and \ref{lem:probsinglepath} below are slight 
modifications of \cite[Lemmas 3.1--3.2]{AGvdHdH} and will be instrumental in the 
proof of Proposition~\ref{prop:mainprop}. The first lemma is concerned with the 
probabilities of the rewiring histories of self-avoiding segmented paths:

\begin{lemma}[\bf{Rewiring histories of self-avoiding segmented paths}]
\label{lem:equalpathprobs}
Fix $t\in\N$, $T\subseteq[t]$ and $\eta,\zeta\in\Conf_H$. Suppose that $x_{[0,t]}$ and 
$y_{[0,t]}$ are two self-avoiding segmented paths in $\eta$ and $\zeta$, respectively, 
of length $t+1$. Then
\begin{equation}
\prob_{\eta,x}\big(A(x_{[0,t]};T)\big) = \prob_{\zeta,y}\big(A(y_{[0,t]};T)\big).
\end{equation}
\end{lemma} 

\begin{proof}
The proof follows the same line of argument as in the proof of \cite[Lemma 3.1]{AGvdHdH}
and uses a coupling between two dynamic configuration models. Let $f$ be a one-to-one 
map from $H$ to itself with the property that it maps $x_i$ to $y_i$ for all $i\in[0,t]$, and 
preserves the edges between two configuration $\eta$ and $\zeta$, i.e., $f(\eta(x)) = \zeta(f(x))$ 
for all $x\in H$. The Markovian coupling $(C_t^x,C_t^y)_{t\in\N_0}$, where $C_0^x = \eta$ 
and $C_0^y = \zeta$, proceeds at every step $t\in\N$ as follows:
\begin{enumerate}
\item 
Choose $k$ edges from $C_{t-1}^x$ uniformly at random without replacement, say 
$\{z_1,z_2\},\dots,$ $\{z_{2k-1},z_{2k}\}$. Choose the edges $\{f(z_1),f(z_2)\},\dots,
\{f(z_{2k-1}),f(z_{2k})\}$ from $C_{t-1}^y$.
\item 
Rewire the half-edges $z_1,\dots,z_{2k}$ uniformly at random to obtain $C_t^x$. Set 
$C_t^y(f(z_i)) = f(C_t^x(z_i))$.
\end{enumerate}
Since under the coupling the event $A(x_{[0,t]};T)$ on $\eta$ is the same as the event 
$A(y_{[0,t]};T)$ on $\zeta$, we get the desired result.
\end{proof}

From this lemma we see that the probability of a specific rewiring history for a self-avoiding 
segmented path does not depend on the path itself nor on the configuration: it only depends on 
$t$ and $T$. In what follows we set $p_t(T) = \prob_{\eta,x}(A(x_{[0,t]},T))$ for which 
$\prob_{\eta,x}(A(x_{[0,t]},T))>0$ for all $T\subset [t]$. When we refer to the modified random 
walk we will use these probabilities as the distribution for the random times $\mathcal{T}$.

The second lemma is concerned with path probabilities for the random walk conditioned on the 
rewiring history:

\begin{lemma}[{\bf Paths estimate given rewiring history}]
\label{lem:probsinglepath}
Suppose that $t = t(n) = [1+o(1)]\,c\log n$ for some $c\in(0,\infty)$, $k = k(n) = \omega((\log n)^2)$ 
and $T =\{t_1,\dots,t_r\}\subseteq [t]$. Let $x_0\cdots x_t\in\SP_t^\eta(x,y;T)$ be a self-avoiding 
segmented path in $\eta$ that starts at $x$ and ends at $y$. Then
\begin{equation}
\prob_{\eta,x}\big(X_{[1,t]} = x_{[1,t]}\mid A(x_{[0,t]};T)\big) 
\geq \frac{1-o(1)}{\ell^r}\prod_{i\in[1,t]\setminus T}\frac{1}{\deg(x_i)}.
\end{equation} 
\end{lemma}

\begin{proof}
The proof follows the same line of argument as the proof of \cite[Lemma 3.2]{AGvdHdH}.
\end{proof}

We continue with the proof of Proposition~\ref{prop:mainprop}. We start by proving the result on 
the tail probabilities of $\tau$, since this is easier.

\begin{proof}[$\rhd$ Proof of \eqref{equ:tautailprob}]
Using \eqref{equ:sa_prob}, we see that
\begin{equation}
\label{equ:sa_tau}
\prob_{\eta,x}(\tau>t)-o(1) \leq \prob_{\eta,x}(\SA_t, \tau>t) \leq \prob_{\eta,x}(\tau>t).
\end{equation}
On the other hand, by considering all possible self-avoiding paths,
\begin{align}
\label{equ:sa_tau_sum}
\prob_{\eta,x}(\SA_t,\tau>t) 
&= \sum_{x_{[0,t]}\in\SP_t^\eta(x;\varnothing)}
\prob_{\eta,x}\big(X_{[1,t]}=x_{[1,t]},A(x_{[0,t]};\varnothing)\big) \nn \\
&=  \sum_{x_{[0,t]}\in\SP_t^\eta(x;\varnothing)}\left(\prod_{i=1}^t\frac{1}{\deg(x_i)}\right)
\prob_{\eta,x}\big(A(x_{[0,t]};\varnothing)\big) \nn \\
&= p_t(\varnothing)\,\pstat_{\eta,x}(\SA_t),
\end{align}
where in the second line we use that
\begin{equation}
\prob_{\eta,x}\big(X_{[1,t]}=x_{[1,t]}\mid A(x_{[0,t]};\varnothing)\big) = \prod_{i=1}^t\frac{1}{\deg(x_i)}
\end{equation}
and in the third line that these are the path probabilities for the random walk without
backtracking in the static model. By following the proof \cite[Eq.\ (2.6)]{AGvdHdH}, we also 
get
\begin{equation}
\prob_{\eta,x}\big(A(x_{[0,t]};\varnothing)\big) = (1-\alpha_n)^{t(t+1)/2}+o(1).
\end{equation}
Combining this with \eqref{equ:sa_prob_stat}, we obtain
\begin{equation}
\prob_{\eta,x}(\SA_t,\tau>t) = (1-\alpha_n)^{t(t+1)/2} + o(1),
\end{equation}
and the claim follows from \eqref{equ:sa_tau}.
\end{proof}

\begin{proof}[$\rhd$ Proof of~\eqref{equ:unstoppedtv}]
Fix $y\in H$. We have
\begin{align}
\prob_{\eta,x}(X_t = y, \SA_t,\tau>t) 
&= \sum_{x_{[0,t]}\in\SP_t^\eta(x,y;\varnothing)} \prob_{\eta,x}\big(X_{[1,t]}=x_{[1,t]},\tau>t\big) \nn \\
&= \sum_{x_{[0,t]}\in\SP_t^\eta(x,y;\varnothing)}\left(\prod_{i=1}^t\frac{1}{\deg(x_i)}\right)
\prob_{\eta,x}\big(A(x_{[0,t]};\varnothing)\big) \\ \nn
&= p_t(\varnothing)\,\prob_{\eta,x}^{\textrm{stat}}(X_t =y,\SA_t).
\end{align}
Using the third line of \eqref{equ:sa_tau_sum}, we obtain
\begin{equation}
\label{equ:dyn_stat_sa}
\prob_{\eta,x}(X_t = y\mid \SA_t,\tau>t) 
= \prob_{\eta,x}^{\textrm{stat}}(X_t =y\mid  \SA_t).
\end{equation}
On the other hand, by partitioning according to $\SA_t$ and $\SA_t^c$ and using that 
$\prob_{\eta,x}(\SA_t) = 1-o(1)$, we obtain
\begin{equation}
\|\prob_{\eta,x}(X_t\in\cdot\mid \tau > t)-\prob_{\eta,x}(X_t\in\cdot\mid \SA_t, \tau>t)\|_{\TV}  = o(1),
\end{equation}
and
\begin{equation}
\|\prob_{\eta,x}^{\mathrm{stat}}(X_t\in\cdot\,)
-\prob_{\eta,x}^{\mathrm{stat}}(X_t\in\cdot\mid \SA_t)\|_{\TV} = o(1).
\end{equation}
Combining these relations with \eqref{equ:dyn_stat_sa}, we obtain
\begin{equation}
\|\prob_{\eta,x}(X_t\in\cdot\mid\tau>t)-\pstat_{\eta,x}(X_t\in\cdot)\|_{\TV}  = o(1).
\end{equation}
Using the results of \cite{BHS} for the random walk without backtracking in the static configuration 
model, we see that if $t = [1+o(1)]c\log n$ with $c \in (0,\cstat)$, then
\begin{equation}
\|\prob_{\eta,x}(X_t\in\cdot\mid\tau>t)-U_H\|_{\TV}  = 1 - o(1),
\end{equation}
while if $t = [1+o(1)]\,c\log n$ with $c \in (\cstat,\infty)$, then
\begin{equation}
\|\prob_{\eta,x}(X_t\in\cdot\mid\tau>t)-U_H\|_{\TV}  = o(1).
\end{equation}
\end{proof}

\begin{proof}[$\rhd$ Proof of~\eqref{equ:stoppedtv}]
Fix $y\in H$ and suppose that $k = k(n) = \omega((\log n)^2)$. Using Lemmas~\ref{lem:equalpathprobs} 
and \ref{lem:probsinglepath},
\begin{align}
&\prob_{\eta,x}(X_t = y, \SA_t, \tau\leq t) \nn \\
&\qquad = \sum_{r=1}^t\sum_{\substack{T\subseteq [1,t]\\|T|=r}}
\sum_{x_{[0,t]}\in\SP(x,y;T)}\prob_{\eta,x}\big(X_{[0,t]}=x_{[0,t]}\mid A(x_{[0,t]};T)\big)\,
\prob_{\eta,x}\big(A(x_{[0,t]};T)\big) \nn \\
&\qquad \geq [1-o(1)] \sum_{r=1}^t\sum_{\substack{T\subseteq [1,t]\\|T|=r}}
p_t(T)\sum_{x_{[0,t]}\in\SP(x,y;T)}\left(\prod_{i\in[t]\setminus T}
\frac{1}{\deg(x_i)}\right)\frac{1}{\ell^r}.
\end{align}
We immediately note that
\begin{equation}
\pmod_{\eta,x}(X_t = y, \SA_t\mid \mathcal{T} = T) 
= \sum_{x_{[0,t]}\in\SP(x,y;T)}\left(\prod_{i\in[t]\setminus T}\frac{1}{\deg(x_i)}\right)\frac{1}{\ell^{|T|}},
\end{equation}
and so
\begin{equation}
\label{equ:dyn_mod_sa}
\prob_{\eta,x}(X_t = y, \SA_t, \tau\leq t) \geq [1-o(1)]\,
\pmod_{\eta,x}(X_t =y, \SA_t, \mathcal{T}\neq \varnothing).
\end{equation}
Using \eqref{equ:modeqstat}, \eqref{equ:sa_prob} and \eqref{equ:sa_tau_sum}, whp in $\eta$ and 
$x$, we also have
\begin{align}
\prob_{\eta,x}(\SA_t,\tau\leq t) 
&= \prob_{\eta,x}(\SA_t) - \prob_{\eta,x}(\SA_t,\tau > t) \nn \\
&\leq \pmod_{\eta,x}(\SA_t) + o(1) - p_t(\varnothing)\,\pstat_{\eta,x}(\SA_t) \nn \\
&= \pmod_{\eta,x}(\SA_t) + o(1) - \pmod_{\eta,x}(\SA_t,\mathcal{T} = \varnothing) \nn \\
&= \pmod_{\eta,x}(\SA_t,\mathcal{T}\neq\varnothing) + o(1).
\end{align}
Combining this with \eqref{equ:dyn_mod_sa} we get, for any $y\in H$,
\begin{equation}
\prob_{\eta,x}(X_t = y\mid \SA_t, \tau\leq t) \geq [1-o(1)]\,
\pmod_{\eta,x}(X_t =y\mid \SA_t, \mathcal{T}\neq \varnothing).
\end{equation}
which in turn gives
\begin{equation}
\label{equ:dyn_mod_sa_tv}
\| \prob_{\eta,x}(X_t \in\cdot \mid \SA_t, \tau\leq t) 
- \pmod_{\eta,x}(X_t \in\cdot\mid \SA_t, \mathcal{T}\neq \varnothing)\|_{\TV} = o(1).
\end{equation}
On the other hand, \eqref{equ:sa_prob} gives
\begin{equation}
\| \prob_{\eta,x}(X_t \in\cdot \mid \SA_t, \tau\leq t) 
- \prob_{\eta,x}(X_t \in\cdot \mid \tau\leq t)\|_{\TV} = o(1)
\end{equation}
and
\begin{equation}
\| \pmod_{\eta,x}(X_t \in\cdot\mid \SA_t, \mathcal{T}\neq \varnothing) 
- \pmod_{\eta,x}(X_t \in\cdot\mid \mathcal{T}\neq \varnothing)\|_{\TV} = o(1).
\end{equation}
Finally, from the latter two relations in combination with \eqref{equ:mod_stat} and 
\eqref{equ:dyn_mod_sa_tv}, we get
\begin{equation}
\| \prob_{\eta,x}(X_t \in\cdot \mid \tau\leq t) - U_H(\cdot)\|_{\TV} = o(1),
\end{equation}
which is the desired result.
\end{proof}



{\small 
\begin{thebibliography}{99}

\bibitem{AGvdHdH}
\textsc{Avena, L.}, \textsc{G\"{u}lda\c{s}, H.}, \textsc{van der Hofstad, R.}, \textsc{den Hollander, F.} 
Mixing times for random walks on dynamic configuration models, 
Preprint. arXiv:1606.07639.

\bibitem{BLPS}
\textsc{Berestycki, N.}, \textsc{Lubetzky, E.}, \textsc{Peres, E.} and \textsc{Sly, A.} (2018).
Random walks on the random graph.
\textit{Ann.\ Probab.} \textbf{46} 456--490.
MR-3758735

\bibitem{BvdHS}
\textsc{Berestycki, N.}, \textsc{van der Hofstad, R.} and \textsc{Salez, J.},
in preparation (2018+).

\bibitem{BHS}
\textsc{Ben-Hamou, A.} and \textsc{Salez, J.} (2017).
Cuttoff for non-backtracking random walks on sparse random graphs.
\textit{Ann.\ Probab.} \textbf{45} 1752--1770. 
MR-3650414


\bibitem{RvdH1}
\textsc{van der Hofstad, R.} (2016).
\textit{Random Graphs and Complex Networks. Vol. I}. Cambridge 
University Press, Cambridge, first edition.
MR-3617364

\end{thebibliography}
}
\end{document}